\documentclass[11pt]{amsart}
\usepackage{amsfonts}
\usepackage{amsmath}
\usepackage{amssymb}
\usepackage{amsxtra}
\usepackage{amsthm}                
\usepackage{subfigure} 
\usepackage[ruled,vlined]{algorithm2e}
\usepackage{comment}
\usepackage{enumerate}
\usepackage{epsfig}
\usepackage{epstopdf}
\usepackage[mathscr]{euscript}
\usepackage{float}
\usepackage[margin=1in]{geometry}  
\usepackage{graphics}
\usepackage{graphicx}
\usepackage{hyperref}
\usepackage[utf8]{inputenc}
\usepackage{latexsym}
\usepackage{mathdots}
\usepackage{microtype}
\usepackage{pict2e}
\usepackage{setspace}
\usepackage{subfigure}
\usepackage{tikz}
\usepackage{times}
\usepackage{wrapfig}

\usepackage{hyperref}
\usepackage[noabbrev]{cleveref}

\usepackage{filecontents}

\usepackage {eufrak}
\usepackage[utf8]{inputenc}
\usepackage{longtable}
\usepackage[lite]{amsrefs}

\renewcommand{\PrintDOI}[1]{\href{http://dx.doi.org/\detokenize{#1}}{doi: \detokenize{#1}}%
	\IfEmptyBibField{pages}{, (to appear in print)}{}}

\usepackage{mathtools}

\theoremstyle{definition}
\newtheorem{theorem}{Theorem}[section]
\newtheorem{lemma}[theorem]{Lemma}

\theoremstyle{definition}
\newtheorem{definition}[theorem]{Definition}
\newtheorem{example}[theorem]{Example}
\theoremstyle{remark}
\newtheorem{remark}[theorem]{Remark}
\numberwithin{equation}{section}

\usepackage{cite}
\numberwithin{equation}{section}

\DeclareMathOperator{\Sym}{Sym}
\DeclareMathOperator{\Span}{Span}
\DeclareMathOperator{\mr}{mir}
 \DeclareMathOperator{\dett}{det}
 \DeclareMathOperator{\inv}{inv}

\begin{document}

\title{An Efficient Algorithm To Compute The Colored Jones Polynomial}

\author{Mustafa Hajij}
\address{Department of Computer Science Engineering,
University of South Florida, 
Tampa, FL USA}
\email{mhajij@usf.edu}

\author{Jesse Levitt}
\address{Department of Mathematics,
University of Southern California, 
Los Angeles, CA USA}
\email{jslevitt@usc.edu }




 \maketitle

The colored Jones polynomial is a knot invariant that plays a central role in low dimensional topology.
We give a simple and an efficient algorithm to compute the colored Jones polynomial of any knot. 
Our algorithm utilizes the walks along a braid model of the colored Jones polynomial that was refined by Armond from the work of Huynh and L\^{e}.
The walk model gives rise to ordered words in a $q$-Weyl algebra which we address and study from multiple perspectives. 
We provide a highly optimized Mathematica implementation that exploits the modern features of the software. 
We include a performance analysis for the running time of our algorithm. Our implementation of the algorithm shows that our method usually runs in faster time than the existing state-of the-art method by an order of magnitude. 

\section{Introduction}
Let $K$ be a knot in $\mathbb{R}^3$ and $N$ be a positive integer.
The colored Jones polynomial (CJP) denoted $J_{N,K}(q)$ and defined in \cite{murakami2001colored,reshetikhin1990ribbon,turaev1988yang} is a 
Laurent polynomial 
with integer coefficients in the variable $q$.
The label $N$ stands for the coloring, i.e., the $N^{th}$ irreducible representation of $\mathfrak{sl}(2,\mathbb{C})$ from which it is calculated.
The polynomial $J_{2,K}(q)$ is the original Jones polynomial \cite{jones1997polynomial}.
The colored Jones polynomial and its generalizations \cite{freyd1985new,turaev2016quantum,jones1990hecke,kauffman1990invariant} play an important role in low-dimensional topology, in particular through its connection to the volume conjecture~\cite{murakami2001colored,kashaev1997hyperbolic,murakami2011introduction}. Since its discovery, the Jones polynomial has lead major discoveries \cite{witten1989quantum,witten19882+} and seen major advances in various areas in low-dimensional topology \cite{le1905aj,khovanov2005categorifications,dasbach2006head,bar1996melvin,le2006colored,garoufalidis2005colored,le2000integrality}.
Moreover, recent years have witnessed considerable developments that established multiple connections between the colored Jones polynomial and number theory.
See for instance~\cite{lee2018trivial,elhamdadi2017pretzel,hajij2016tail,beirne2017q,lovejoy2013bailey,armond2011rogers,hajij2017colored,bataineh2016colored}.
The coefficients of the colored Jones polynomial have been proven to give rise to Ramanujan theta and false theta identities~\cite{armond2011rogers,elhamdadi2017foundations}. See also \cite{futer2012guts} and the references therein for more about the history and development of the colored Jones polynomial. 

Computing the colored Jones polynomial is a highly non-trivial task. Much of the literature on computing the colored Jones polynomial is devoted to giving quantum algorithms of this polynomial. The existence of an efficient algorithm for approximating was implied in the work of Freedman, and  Kitaev \cite{freedman2002simulation}. Later an explicit quantum algorithm for approximating the Jones polynomial was given in \cite{aharonov2009polynomial} and extended later in \cite{wocjan2006jones,aharonov2011bqp}. More on the quantum computation of the Jones polynomial can be found in \cite{kauffman2002quantum,passante2009experimental}.

One of the earlier classical algorithms to compute the colored Jones polynomial was given in \cite{masbaum19943} by Masbaum and Vogel where the skein theory associated to the Kauffman bracket skein module was utilized. This algorithm can be considered as an extension of the Kauffman bracket which in turn can be used in algorithms to compute the original Jones polynomial. Masbaum and Vogel's algorithm however relies on certain diagrammatic manipulations that require special handling for each knot making it hard to obtain an efficient general implementation. In \cite{garoufalidis2005colored} the $q$-holonomicity of the colored Jones polynomial was proven and this in turn can be used to compute the colored Jones polynomial. Bar-Natan’s Mathematica package KnotTheory \cite{KA} implements this to compute the colored Jones polynomial. The algorithm however is mostly feasible for knots with small crossing number and color $N < 9$. 
The other commonly used publicly available implementation for the Jones polynomial that we are aware of is SnapPy \cite{SnapPy}, but this implementation however only handles the Jones polynomial, i.e., when $N=2$. 

Explicit formulas for  of the colored Jones polynomial of the double twist knots can be found in \cite{masbaum2003skein} and for torus knots in \cite{morton1995coloured}. Moreover, a difference equation for torus knots is given in \cite{hikami2004difference}. The complexity of the Jones polynomial  of alternating links is studied in \cite{jaeger1990computational}. More on the computational complexity of the Jones polynomial and its generalization can be found in \cite{freedman1998p}.
 
 
In this article we give an efficient classical algorithm to compute the colored Jones polynomial for any knot based on the quantum determinant formula suggested by Huynh and L\^{e}~\cite{huynh2005colored}.  In particular we consider a walk along a braid interpretation of the evaluation of this quantum determinant that was developed by Armond~\cite{armond2014walks}, in light of Jones' interpretation of the Burau Representation~\cite{jones1987Hecke}.
This walk model gives rise to an ordered word in a $q$-Weyl algebra which is studied from multiple perspectives. The algorithm converts each word to a standard word defined in this paper and then evaluates each word to a Laurent polynomial. To minimize the number of words needed for evaluation, we utilize a structural theorem regarding the set of walks on braids. Along with our algorithm we provide a highly optimized Mathematica ~\cite{Mathematica} implementation that exploits several modern features and functionalities in Mathematica. 
\section{Preliminaries}


We give a quick review here of braid groups as it will be needed in later sections.

\subsection{Braid Group}
The input of the colored Jones polynomial algorithm we present here is a braid whose closure forms a knot.
Alexander's theorem~\cite{alexander1923lemma} assures that any knot can be realized as the closure of a  braid in this manner.


Let $D^3$ denote the $3$ manifold with boundary $[0,1]^3$. Fix $m$ points on the top of $D^3$ and $m$ points on the bottom. A \textit{braid} on $m$ strands is a curve $\beta_m$ embedded in $D^3$ and decomposed into $m$ arcs such it meets $D^3$ orthogonally in exactly $2m$ points and where no arc intersects any horizontal plane more than once. A braid is usually represent a planar projection or a \textit{braid diagram}.
 In the braid diagram we make sure that each crossing the over strand is distinguishable from the under-strand by creating a break in the under-strand. Figure~\ref{braid} shows an example of a braid diagram on $3$ strands.

\begin{figure}[H]
  \centering
   {\includegraphics[scale=0.4]{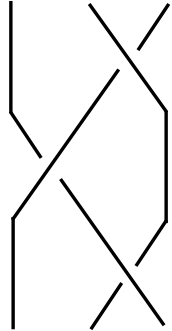}
    \caption{An example of a braid diagram on $3$ strands.}
  \label{braid}}
\end{figure}

The set of all braids $B_m$ has a group structure with multiplication as follows. Given two $m$ strand braids $\beta_1$ and $\beta_2$ the product of these braids, $\beta_1\cdot \beta_2$ is the braid given the by vertical concatenation of $\beta_1$ on top of $\beta_2$. As in Figure~\ref{product}.
\begin{figure}[H]
  \centering
   {\includegraphics[scale=0.33]{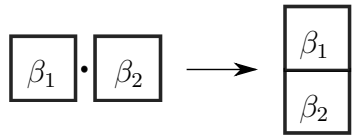}
      \put (-123,20) {\large{$\beta_1$}}
    \put (-90,20) {\large{$\beta_2$}}
     \put (-20,31) {\large{$\beta_1$}}
     \put (-20,7) {\large{$\beta_2$}}
    \caption{The product of two braids.}
  \label{product}}
\end{figure}

The group structure of $B_m$ follows directly from this. The braid group $B_m$ on $m$ strands can be described algebraically in terms of generators and relations using Artin's presentation. In this presentation, the group $B_m$ is given by the generators : 
\begin{equation*}
\sigma_1,\ldots,\sigma_{n-1},\sigma_1^{-1},\ldots,\sigma_{n-1}^{-1}, 
\end{equation*}
subject to the relations:
\begin{enumerate}
\item For all $1\leq  i < n$:   $\sigma_i \sigma_i^{-1} = e=\sigma_i^{-1} \sigma_i $.
\item For $|i-j|>1$: $\sigma_i \sigma_j= \sigma _j \sigma _i$. 
\item For all $i<m-1$:  $\sigma_i \sigma_{i+1} \sigma_i = \sigma_{i+1}\sigma_i \sigma_{i+1}$
\end{enumerate}
The correspondence between the pictorial definition of the braid group and the algebraic definition is given by sending the generator $\sigma_i$ to the picture illustrated in Figure~\ref{generator}.

\begin{figure}[H]
  \centering
   {\includegraphics[scale=0.3]{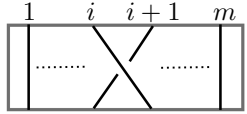}
        \put (-84,35) {\small{$1$}}
    \put (-60,35) {\small{$i$}}
    \put (-45,35) {\small{$i+1$}}
    \put (-12,35) {\small{$m$}}
    \caption{The braid group generator $\sigma_i$.}
  \label{generator}}
\end{figure}

The \textit{braid closure} $\hat{\beta}$ of a braid $\beta$ is given by joining the $m$ points
on the top of $\beta$ to the $m$ points on the bottom by parallel arcs as follows:
\begin{figure}[H]
  \centering
   {\includegraphics[scale=0.45]{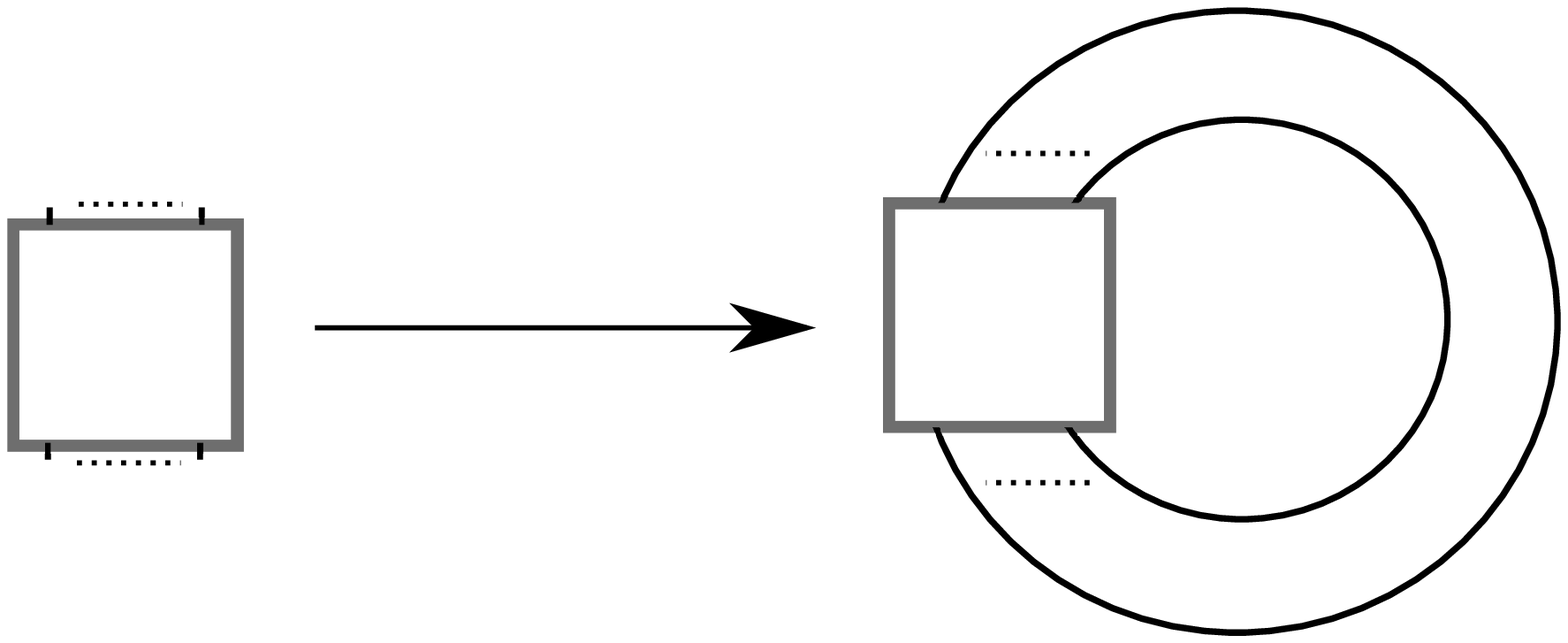}
   \put (-215,42) {$\beta$}
  \put (-90,42) { $ \beta$ }
  \put (-60,-15) { $\hat{ \beta}$ }
    \caption{The braid $\beta$ and its closure $ \hat{\beta}$.}
  \label{closure}}
\end{figure}

\section{Walks on Braids}
\label{sec:Walks}
In~\cite{huynh2005colored} Lee suggested a realization of the colored Jones polynomial as the inverse of the quantum determinant of an almost quantum matrix.
The entries of this matrix live in the $q$-Weyl algebra. The quantum determinant version of the colored Jones polynomial admits a walks along a braid model described by Armond in~\cite{armond2011rogers}. In this section we review this walk model and define the necessary terms that will be utilized in the algorithm section.

Recall that all knots can be realized as the closure of a braid~\cite{alexander1923lemma}. Moreover, given a knot, this braid can be computed using Yamada-Vogel's Algorithm~\cite{yamada1987minimal,vogel1990representation}. The algorithm utilized here takes a braid as input,which is uniquely determined by the braid sequence $\alpha$, described next. 

Let $\beta$ be a braid in $B_m$ given by the braid word:
\begin{equation}
\beta=\sigma_{i_1}^{\epsilon_1}\sigma_{i_2}^{\epsilon_2}\ldots\sigma_{i_k}^{\epsilon_k}.
\end{equation}

The \textit{ braid sequence} $\alpha$ of the braid $\beta$    is a finite sequence $\alpha = (\alpha_1,\alpha_2,\ldots,\alpha_k)$ of pairs $\alpha_j=(i_j,\epsilon_j)$,  $1 \leq i_j < m$  and $\epsilon_j = \pm 1$. For instance the braid sequence of the braid $\sigma_1^{-1} \sigma_2 \sigma_1^{-1} \sigma_2$ is $\beta(\alpha) = ((1, -1),(2, +1),(1, -1),(2, +1))$. Conversely, a sequence  $\alpha = (\alpha_1,\alpha_2,\ldots,\alpha_k)$ of pairs $\alpha_j=(i_j,\epsilon_j)$, where $1 \leq i_j < m$ and $\epsilon_j = \pm 1$ gives rise to a braid  $\sigma_{i_1}^{\epsilon_1}\sigma_{i_2}^{\epsilon_2}\ldots\sigma_{i_k}^{\epsilon_k}$ in $B_m$. We recommend the reader to the work of T. Gittings~\cite{gittings2004minimum} or the tables of C. Livingston~\cite{livingston} for collections of minimal braid sequences.

We will denote by $\beta(\alpha)$ the braid associated to the sequence $\alpha = (\alpha_1,\alpha_2,\ldots,\alpha_k)$. For the rest of this paper we will refer to the braid $\beta$ and its associated braid sequence interchangeably. Moreover, we will refer to the pair $\alpha_j=(i_j,\epsilon_j)$ as a crossing in the braid. 

A $\textit{path}$ on a braid $\beta(\alpha)$ from a strand $i$, counting up from left to right as in Figure~\ref{generator}, to a strand $j$, is defined as follows. We start at the bottom at the $i$ strand of the braid and we march to the top. Whenever arriving at an over-crossing we can either jump down to the lower strand or continue along the same strand. We continue in this manner until we reach the top of the braid at the $j$ strand. Note that when a path goes from the bottom to the top on a braid $\beta(\alpha)$ it passes some collection of crossings $\alpha_i$ of $\beta$. At a crossing $\alpha_i$ we encode this passage by the following conventions:       

\begin{enumerate}
\item If a path jumps down at $\alpha_i$, we assign that crossing the weight $a_{i}^{\epsilon_i}$.

\item If a path follows the bottom strand at $\alpha_i$, we assign that crossing the weight $b_{i}^{\epsilon_i}$.

\item If a path follows the top strand at $\alpha_i$, we assign that crossing the weight $c_{i}^{\epsilon_i}$.
\end{enumerate}
Figure~\ref{paths_weight} illustrates the three types possible behavior of a path at a crossing $\alpha_i=(j,+1)$ and the weights assigns to local path in each case. 
\begin{figure}[H]
  \centering
   {\includegraphics[scale=0.7]{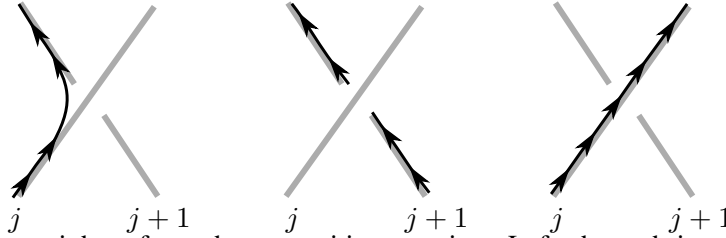}
   \put (-260,-10) {$j$}
   \put (-215,-10) {$j+1$}
   \put (-155,-10) {$j$}
   \put (-110,-10) {$j+1$}
    \put (-55,-10) {$j$}
    \put (-10,-10) {$j+1$}
    \caption{The weights of a path at a positive crossing. Left, the path jumps down, this assigns the weight \textbf{$a_i^+$}. Middle, the path follows the bottom strand, this assigns the weight \textbf{$b_i^+$}. Right, the path follows the top strand, this assigns the weight \textbf{$c_i^+$}.}
  \label{paths_weight}}
\end{figure}

The weights $a$, $b$ and $c$ will be given additional meaning in section~\ref{algebra}. The \textit{weight} of a single strand path is the product of the weights of its crossings, the \textit{weight of a path} is the product from bottom right to left of the weights of each individual strand path. 
Notice that the weights are not allowed to commute in general. This is elaborated on in section~\ref{algebra}.

\begin{remark}
It is important to notice that we read the braid in two different directions. When we read a path on a braid we read it from bottom to top. Whereas when we read the braid word as product of braid group generators we read that from top to bottom. See example~\ref{ex:Fig8} for an illustration.
\end{remark}
A \textit{walk} $W$ along the braid $\beta$ in $B_m$ consists of the following data :

\begin{enumerate} 

\item A set $J \subset  \{1,\ldots,m\}$.
\item A permutation $\pi$  of $J$.

\item A collection of paths $\mathcal{P}$ on $\beta$ with exactly one path in $\mathcal{P}$ from strand $j$ to strand $\pi (j)$, for each $j \in  J$.
\end{enumerate}

We denote by $\inv(\pi)$ the number of inversions in the permutation $\pi$, i.e. the number of pairs $i < j$ such that $\pi (i) > \pi(j)$. Then every walk is assigned a weight defined as $(-1)(-q)^{|J|+\inv(\pi)}$ times the product of the weights of the paths in the collection $\mathcal{P}$. A  walk is said to be \textit{simple} if no two paths in the collection $\mathcal{P}$ traverse the same point on the braid. Simple walks are desirable for computational reasons and we will consider only them  for algorithmic efficiency. A \textit{stack of walks} is an ordered collection of walks. Finally, \textit{the weight of a stack} is the product of the weights of its walks.

\section{Algebra of the deformed Burau matrix}
\label{algebra}
Let $ \mathcal{R}= \mathbb{Z}[q^{\pm1}]$. For each braid $\beta=\sigma_{i_1}^{\epsilon_1}\sigma_{i_2}^{\epsilon_2}\cdots\sigma_{i_k}^{\epsilon_k}$ we define an $\mathcal{R}$-algebra   $\mathcal{A}_\epsilon$, where  $\epsilon$ is the sequence $ (\epsilon_1,\epsilon_2,\ldots,\epsilon_k)$, as follows. We associate to every $\epsilon_i$ a the letters : $a_{i}^{\epsilon_i}$, $b_{i}^{\epsilon_i}$ and $c_{i}^{\epsilon_i}$. The algebra $\mathcal{A}_\epsilon$ is the $\mathcal{R}$-algebra freely generated by the set $\mathcal{L}_{\epsilon}\coloneqq \{a_{i}^{\epsilon_i},b_{i}^{\epsilon_i},c_{i}^{\epsilon_i}\}_{i=1}^k$ subject to the following commutation relations:
\begin{align}
\label{relation1}
a_{i}^{+} b_{i}^{+} &= b_{i}^{+} a_{i}^{+} ,& a_{i}^{+} c_{i}^{+} &= qc_{i}^{+}a_{i}^{+},& b_{i}^{+}c_{i}^{+} &= q^2c_{i}^{+}b_{i}^{+},\\
\label{relation2}
a_{i}^{-}b_{i}^{-} &= q^2b_{i}^{-}a_{i}^{-},& c_{i}^{-}a_{i}^{-} &= qa_{i}^{-}c_{i}^{-},& c_{i}^{-}b_{i}^{-} &= q^2b_{i}^{-}c_{i}^{-}
\end{align}
Here $q$ can be thought of as a \textit{skew-commutator} for each $\mathcal{R}-$algebra, but for any two elements $x_i,y_j\in \mathcal{A}_{\epsilon}$ where $i\neq j$ we have: \begin{equation}
\label{relation3}
x_i y_j=y_j x_i
\end{equation}

The relationship between the walks introduced in the previous section and elements of the algebra $\mathcal{A}_{\epsilon}$ will be made explicit in section~\ref{CJP}. For the rest of the paper we fix  $\epsilon$ to be the sequence $(\epsilon_1,\epsilon_2,\ldots,\epsilon_k)$ and refer to the algebra associated to $\epsilon$ by $\mathcal{A}_{\epsilon}$ as defined above.

\subsection{Right Quantum Words in $\mathcal{A}_{\epsilon}$}
\label{sssec:RQWord}

Relations~\ref{relation1},~\ref{relation2} and~\ref{relation3} are utilized in our algorithm to convert any word in the algebra \(\mathcal{A}_{\epsilon}\) to a standardized word that will facilitate the computation of the colored Jones polynomial. 


 A \textit{word} in $\mathcal{A}_{\epsilon}$ is a finite product of elements from $\mathcal{L}_{\epsilon}$. A monomial in $\mathcal{A}_{\epsilon}$ is a product $ \gamma W$ where $\gamma \in \mathcal{R}$ and $W$ is a word in $\mathcal{A}_{\epsilon}$. A word $W$ in $\mathcal{A}_{\epsilon}$ is said to be \textit{a right quantum} if it has the form: 
%
\begin{equation*}
W= W^{+} W^{-}
\end{equation*} 
where
\begin{equation}
\label{plus}
W^{+}=
\left(b_{i_1}^{+}\right)^{s_1}\left(c_{i_1}^{+}\right)^{r_1}\left(a_{i_1}^{+}\right)^{d_1} \cdots \left( b_{i_u}^{+}\right)^{s_u}\left(c_{i_u}^{+}\right)^{r_u}\left(a_{i_u}^{+}\right)^{d_u},
\end{equation}
\begin{equation}
\label{min}
W^{-}=
\left(b_{j_1}^{-}\right)^{s^{\prime}_1}\left(c_{j_1}^{-}\right)^{r^{\prime}_1}\left(a_{j_1}^{-}\right)^{d^{\prime}_1} \cdots \left( b_{j_v}^{-}\right)^{s^{\prime}_v}\left(c_{j_v}^{-}\right)^{r^{\prime}_v}\left(a_{j_v}^{-}\right)^{d^{\prime}_{v}}
\end{equation}
such that $1 \leq i_1< i_2 < \cdots < i_u \leq k $, and $ 1\leq j_1< \cdots < j_{v} \leq k $. By convention the empty word will be assumed to be right quantum.  Algorithm~\ref{MB} summarizes the first step in the conversion of a word to a right quantum word. 

~

\begin{algorithm}[H]
    \label{MB}
		\KwIn{ A list of exponents of $a_i's$ ordered by crossing number\;
A list of exponents of $b_j's$ ordered by crossing number\;
A list of exponents of $c_k's$ ordered by crossing number\;
}
		\KwOut{ The right quantum form of the word with those exponents.}
        \vspace{1mm}
Start with an empty word $W$\;
For each exponent $i_j$ in the $b_j$ exponent list, append $b_j$ to the right of the word $W,$ $i_j$ times\;
For each exponent $i_j$ in the $c_j$ exponent list, append $c_j$ to the right of the word $W,$ $i_j$ times\;
For each exponent $i_j$ in the $a_j$ exponent list, append $a_j$ to the right of the word $W,$ $i_j$ times\;
		\caption{Monomial Builder (MB)}
\end{algorithm}

Now let $U$ be an arbitrary word in $\mathcal{A}_{\epsilon}$.  Using the relations~\ref{relation1},~\ref{relation2} and~\ref{relation3} we can write:
\begin{equation}
\label{quantum}
U=  q^{z} U^{\prime},
\end{equation}
where $U^{\prime}$ is a right quantum word and $z$ is an integer. We call the word $U^{\prime}$ the quantum word associated with the word $U$. Using this convention, we define the \textit{word vector} of $\mathcal{V}(U)$ as the sequence of integers :
\begin{equation}
\mathcal{V}(U)=(z,s_1,r_1,d_1,\cdots,s_u,r_u,d_u,s^{\prime}_1,r^{\prime}_1,d^{\prime}_1,\cdots,s^{\prime}_u,r^{\prime}_v,d^{\prime}_v).
\end{equation}


Note here that the exponents $r_i$, $s_i$ and $d_i$ are simply the number of times each variable $b^{+}_i$, $c^{+}_i$ and $d^{+}_i$ occurs in the word $U$. Hence these numbers can be computed simply without any consideration of the algebra relations. The same holds for determining the exponents $r^{\prime}_i$, $s^{\prime}_i$ and $d^{\prime}_i$. On the other hand, in order to compute the exponent $z$ the word $U$ must be converted to right quantum form using the algebra relations. We will come back to this fact later while describing the main algorithm.

\begin{remark}
Since any monomial is a scalar multiple of a word, the definition of quantum word can be also defined on monomials. For this reason the technical difference between words and monomials is not essential and  in our discussion below we use these two terms interchangeably.  
\end{remark}

\begin{remark}
Note that the weight of an arbitrary walk along a braid is simply a word in $\mathcal{A}_{\epsilon}$. For our purposes, we are not interested in all words in $\mathcal{A}_{\epsilon}$, instead we are merely concerned about words that can be realized as a weight of walk or a product of such words. This fact will be utilized in Section~\ref{Efficient Evaluation} in order to reduce the calculations needed for evaluation of these words as elements in $\mathcal{R}$.
\end{remark}

\section{The Evaluation Map $\mathcal{E}_N$ }
\label{map}
Here we introduce the $N$-evaluation map $\mathcal{E}_N$ that operates on  elements in $\mathcal{A}_\epsilon$ and returns a polynomial in $\mathcal{R}$. Since the purpose of this paper is to give an algorithm to compute the colored Jones polynomial we will solely give the axioms that are necessary to calculate the map $\mathcal{E}_N$. The reader further interested in this map is referred to the paper~\cite{lee2010introduction} for an equivalent set of axioms and their full construction. 

\begin{definition}
\label{right quantum}
For $N\geq 1,$ the function $\mathcal{E}_N : \mathcal{A}_\epsilon \longrightarrow \mathcal{R}$ is defined via the following axioms:

\begin{enumerate}
\item For any $f,g \in \mathcal{A}_\epsilon$,
\begin{equation*}
\mathcal{E}_N(f+g)=\mathcal{E}_N(f)+\mathcal{E}_N(g).
\end{equation*}
\item For any $c \in \mathcal{R}$ and $f \in \mathcal{A}_{\epsilon}$,
\begin{equation*}
\mathcal{E}_N(cf)=c\,\mathcal{E}_N(f).
\end{equation*}
\item For any two monomials $f,g  \in \mathcal{A}_{\epsilon}$ that are \emph{separable}, which occurs when the monomial $f$ only contains the letters $a_i,b_i,c_i$ with $i\in I$ and the monomial $g$ has only letters $a_j, b_j, c_j$ with $j\in J$ where $I\cap J=\emptyset$, we have 
\begin{equation*}
\mathcal{E}_N(fg)=\mathcal{E}_N(f)\cdot\mathcal{E}_N(g).
\end{equation*}
\item $\mathcal{E}_N \left(\left(b_i^+\right)^s \left(c_i^+\right)^r \left(a_i^+\right)^d\right) = q^{r(N-1-d)} \displaystyle{\prod_{j=0}^{d-1} \left(1 - q^{N-1-r-j}\right)}.$
\item $\mathcal{E}_N \left(\left(b_i^-\right)^s \left(c_i^-\right)^r \left(a_i^-\right)^d\right)= q^{-r(N-1)} \displaystyle{\prod_{j=0}^{d-1}  \left(1 - q^{r+j+1-N} \right)}$.
\end{enumerate}
\end{definition}
\noindent Now let $C $ be an arbitrary element in $\mathcal{A}_{\epsilon}$. We are interested in the explicit evaluation of $\mathcal{E}_N(C)$. To this end, observe first that the element $C$ can be written in the form $C=W_1+\cdots +W_t$
where $W_i$ is a monomial in $\mathcal{A}_{\epsilon}$  for $1 \leq  i \leq t$. The computation of the colored Jones polynomial in our algorithm relies essentially on the computation of $\mathcal{E}_N(C)$. In the case of the colored Jones polynomial the element $C$ is not just an arbitrary element in the algebra $\mathcal{A}_{\epsilon}$. Specifically, the set  $\mathcal{W}\coloneqq\{W_i \}_{i=1}^t$ possesses a structure with respect to the evaluation map $\mathcal{E}_N$ that allows for its efficient computation. We will introduce this structure along with the method to evaluate $\mathcal{E}_N$ on the special types of elements $C \in \mathcal{A}_{\epsilon}$ in Section~\ref{E section} after we give the definition of the colored Jones polynomial in terms of the evaluation map. 

\section{Quantum determinant}
Let $a,b,c,d$ be elements of the noncommutative ring $\mathcal{B}.$
A $2 \times 2 $ matrix $
\begin{bmatrix}
 a & b \\
c & d \\
\end{bmatrix}
\in M_{2}(\mathcal{B})$
 is said to be right quantum if :
\begin{enumerate}
\item $ac = qca$
\item $bd = qdb$
\item $ad = da+qcb-q^{-1}bc$
\end{enumerate}
An $m\times m$ matrix is said to be right quantum if all its $2 \times 2$ submatrices are right quantum. If $A$ is a right quantum matrix then the quantum determinant of $A$ is given by :
\begin{equation}
\label{eqn:qdet}
\dett_q(A)=\sum_{\pi \in \Sym(m)} (-q)^{\inv(\pi)} a_{\pi(1),1}\cdots a_{\pi(m),m} 
\end{equation}
where $\inv(\pi)$ is the number of inversions in the permutation $\pi$.

Let $A$ be an $m\times m$ matrix and let $\emptyset\neq J \subseteq \{1,\ldots,m\}$. The matrix $A_J$ is the $J\times J$ submatrix of $A$ obtained by selecting the rows and columns of $A$ that correspond to the set $J$. Note that if $A$ is right quantum then $A_J$ is also right quantum. If $A$ is a right quantum matrix then define $C_A$ by 
\begin{equation}
\label{eqn:altSum}
C_A\coloneqq\sum_{\emptyset \neq J \subset \{1,\ldots,m\} } (-1)^{|J|-1 } \dett_q(A_J)
\end{equation}

\section{Deformed Burau Matrix}
\label{ssec:deformedB}
The deformed Burau matrix associated to a braid is defined similar to the Burau Matrix~\cite{bmma82x}. Define the matrices

\begin{align*}
&&
S^+&\coloneqq
\begin{bmatrix}
 a^+ & b^+ \\
c^+ & 0 \\
\end{bmatrix}&
S^-&\coloneqq
\begin{bmatrix}
 0 & c^- \\
b^- & a^- \\
\end{bmatrix}
&&
\end{align*}

  
To every braid generator $\sigma_{i_j}^{\epsilon_j}$ in Artin's standard presentation we associate an $m \times m$ right-quantum matrix $A_j$ which is the identity matrix except at the submatrix of rows $i_j$, $i_j + 1$ and
columns $i_j$, $i_j + 1$ which we replace by $S_{j}^{\epsilon_j}$. The matrix $S_{j}^{\epsilon_j}$ is the same as the matrix $S^{\epsilon}$ but $a^{\epsilon_j}$, $b^{\epsilon_j}$ and $c^{\epsilon_j}$ are replaced with $a_{j}^{\epsilon_j}$, $b_{j}^{\epsilon_j}$ and $c_{j}^{\epsilon_j}$ respectively.
The matrix $A_j$ is called the \textit{deformed Burau matrix associated with the generator $\sigma_{i_j}^{\epsilon_j}$}.
For a braid $\beta$ given as   $\sigma_{i_1}^{\epsilon_1}\sigma_{i_2}^{\epsilon_2}\cdots\sigma_{i_k}^{\epsilon_k}$, the deformed Burau matrix $\rho(\beta)$ is defined by the multiplication of its corresponding deformed Burau matrices:  $A_1A_2\cdots A_k$.  The reduced deformed Burau matrix $\rho^{\prime}(\beta)$ is obtained from $\rho(\beta)$  by dropping both the first row and first column. Note that $\rho(\beta)$ is a right quantum matrix, so $\rho^{\prime}(\beta)$ is as well.

Looking back at Figure~\ref{paths_weight}, one can see how for a positive crossing $\alpha_i=(j,+1)$ the $a^+_i$ weight corresponds to a path moving from the bottom of the diagram to the top by following from $j^{th}$ position on the overcrossing strand and jumping to the $j^{th}$ position on the undercrossing strand, the $b^+_i$ weight corresponds to a path that moves from the $(j+1)^{st}$ position to the $j^{th}$ position by following the undercrossing strand and the $c^+_i$ weight corresponds to a path that moves from the $j^{th}$ position to the $(j+1)^{st}$ position by following the overcrossing strand. Similarly, for negative crossings, the matrix elements of the matrix, $A_j$, correspond to the transition from the bottom strand (column) of a crossing to the top strand (row).


\section{The Colored Jones Polynomial}
\label{CJP}
In this section we give the definition of the colored Jones polynomial using the Burau representation we described in the previous section. We recall some facts about the Jones and colored Jones polynomial first. The Jones polynomial is a Laurent polynomial knot invariant in the variable $q$ with integer coefficients. The Jones polynomial generalizes to an invariant $J_{N,K}(q)\in \mathbb{Z}[q^{\pm 1}]$ of a knot $K$ colored by the $N^{th}$ irreducible representation of $\mathfrak{sl}(2,\mathbb{C})$ and normalized so that $J_{N,O}(q)=1$, where $O$ denotes the unknot. Note that this is the unframed normalized version of the colored Jones polynomial. The original Jones polynomial corresponds to the case $N=2$. In~\cite{huynh2005colored} it was shown that the colored Jones polynomial of a knot can be computed in terms of the quantum determinant of the deformed Burau representation of a braid $\beta$ with $\hat{\beta}=K$. We recall the statement of this theorem here since it will be utilized in our algorithm.  

\begin{theorem}[{\cite[Theorem~1]{huynh2005colored}}]
\label{main theorem}
Suppose the closure in the standard way of the $m$-strand braid $\beta (\alpha)$ is a knot $K$. Then for any positive integer $N \geq 1 $ we have


\begin{equation*}
J_{N,K}(q) = q^{(N-1)(\omega(\beta)-m+1)/2} \sum_{i=0}^{\infty} \mathcal{E}_N\left( C_{q\rho^{\prime}(\beta)}^i\right)    
\end{equation*}
where $\omega(\beta)=\sum\epsilon_j$ is the writhe of the knot and $m$ is the number of strands in its braid representation.
\end{theorem}

\begin{remark}
\label{cody thm}
In~\cite{armond2011rogers} Armond interpreted the polynomial $ C_{\rho^{\prime}(\beta)}$ from \cref{eqn:altSum} as a sum of the weights of walks on $\beta(\alpha)$ with $J  \subset\{2,\ldots,m\}$. Noting that this deformed Burau representation corresponds to removing the first strand as either a starting or ending point, but allowing it to be traversed midbraid. In this interpretation, it is natural to understand higher powers $C_{\rho^{\prime}(\beta)}^i$ as a stack of superimposed walks with $J  \subset\{2,\ldots,m\}$. We elaborate on the idea of a stack of walks further in section~\ref{Efficient Evaluation}.
\end{remark}



\begin{example}
\label{ex:Fig8}
 Let $\beta = ((1, -),(2, +),(1, -),(2, +))$ be the braid $\sigma_1^{-1} \sigma_2 \sigma_1^{-1} \sigma_2$ in $B_3$ ( see Figure~\ref{braid8}). 

\begin{figure}[h]
  \centering
   {\includegraphics[scale=0.38]{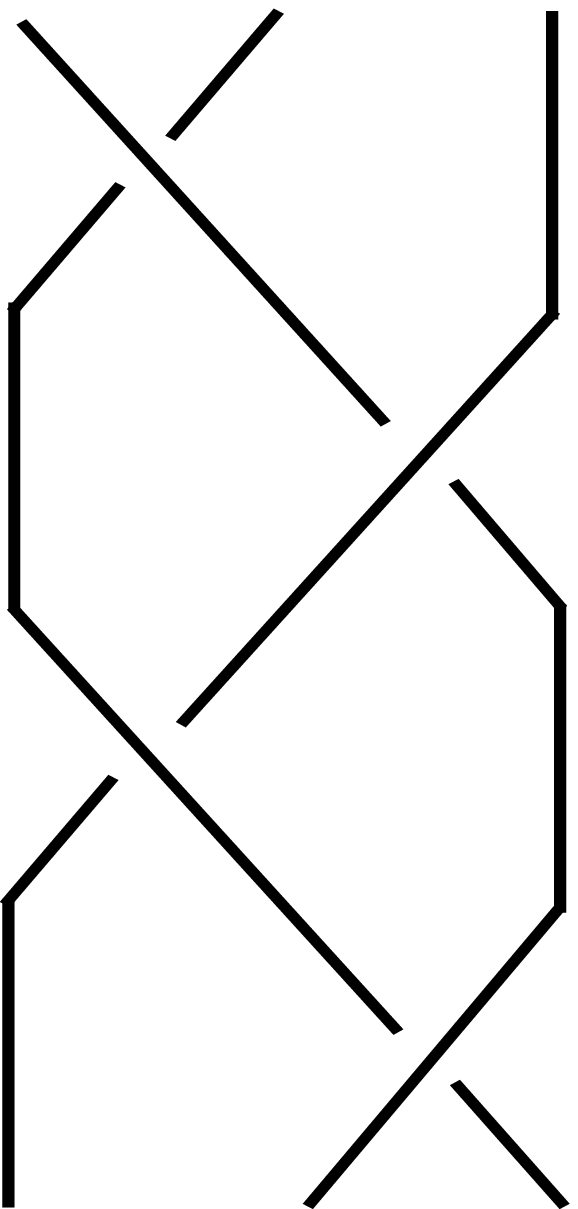}
    \caption{The braid $\sigma_1^{-1} \sigma_2 \sigma_1^{-1} \sigma_2$ in $B_3$ .}
  \label{braid8}}
\end{figure}
 
 The closure of $\beta$ gives the figure eight knot. Moreover, $m = 3$ and $\omega(\beta) = 0$. Computing the deformed Burau representation of $\beta$ gives the right-quantum matrix:

\begin{small}
\begin{eqnarray*}
\rho(\beta)&=& \rho(\sigma_1^{-1})\rho(\sigma_2)\rho(\sigma_1^{-1})\rho(\sigma_2) \\&=&\left( \begin{array}{ccc}

0 & c_1^- & 0 \\
b_1^{-}& a_1^- & 0 \\
0 &0 & 1 \\
\end{array}\right)\left( \begin{array}{ccc}
1 & 0 & 0 \\
0& a_2^+ & b_2^{+} \\
0 &c_2^+ & 0 \\
\end{array}\right)\left( \begin{array}{ccc}
0 & c_3^- & 0 \\
b_3^-& a_3^- & 0 \\
0 &0 & 1 \\
\end{array}\right)\left( \begin{array}{ccc}
1 & 0 & 0 \\
0& a_4^+ & b_4^{+} \\
0 &c_4^+ & 0 \\
\end{array}\right)  \\&=&  
\left(
\begin{array}{ccc}
 \vspace{1mm}
c^{-}_1a^+_2b^{-}_3 & c^{-}_1b^+_2c_4+c^{-}_1a^+_2a^{-}_3a^+_4 & c^{-}_1a^+_2a^{-}_3b^+_4 \\
 \vspace{1mm}
a^{-}_1a^+_2b^{-}_3 & a^{-}_1b^+_2c^+_4+b^{-}_1c^{-}_3a^+_4+a^{-}_1a^+_2a^{-}_3a^+_4 & b^{-}_1c^{-}_3b^+_4+a^{-}_1a^+_2a^{-}_3b^+_4 \\
 c^+_2b^{-}_3 & c^+_2a^{-}_3a^+_4 & c^+_2a^{-}_3b^+_4 \\
\end{array}
\right)
\end{eqnarray*}
\end{small}

Hence the reduced Burau matrix is given by :

\begin{small}
\begin{equation}
\rho^{\prime}(\beta)=
\left(
\begin{array}{cc}
\vspace{1mm}
a^{-}_1b^+_2c^+_4+b^{-}_1c^{-}_3a^+_4+a^{-}_1a^+_2a^{-}_3a^+_4 & b^{-}_1c^{-}_3b^+_4+a^{-}_1a^+_2a^{-}_3b^+_4 \\
c^+_2a^{-}_3a^+_4 & c^+_2a^{-}_3b^+_4 
\end{array}
\right)
\end{equation}
\label{eqn:reducedBurau}
\end{small}

Using the quantum determinant formula one obtains the following summation of 
walks:

\begin{align}
C_{q\rho^{\prime}(\beta)}^1=&\; q^3 c_{2}^{+}a_{3}^{-}a_{4}^{+}a_{1}^{-}a_{2}^{+}a_{3}^{-}b_{4}^{+} - q^2 a_{1}^{-}a_{2}^{+}a_{3}^{-}a_{4}^{+}c_{2}^{+}a_{3}^{-}b_{4}^{+} + q a_{1}^{-}a_{2}^{+}a_{3}^{-}a_{4}^{+}\nonumber\\
& + q^3 c_{2}^{+}a_{3}^{-}a_{4}^{+}b_{1}^{-}c_{3}^{-}b_{4}^{+} - q^2 b_{1}^{-}c_{3}^{-}a_{4}^{+}c_{2}^{+}a_{3}^{-}b_{4}^{+} + q b_{1}^{-}c_{3}^{-}a_{4}^{+} \label{eqn:exampleWalks}\\
&- q^2 a_{1}^{-}b_{2}^{+}c_{4}^{+}c_{2}^{+}a_{3}^{-}b_{4}^{+} + q c_{2}^{+}a_{3}^{-}b_{4}^{+} + q a_{1}^{-}b_{2}^{+}c_{4}^{+}\nonumber 
\end{align}

Thus the colored Jones polynomial is given by:

\begin{equation*}
J_{N,\hat{\beta}}(q)=q^{(N-1)}\sum_{i=0}^{\infty}  \mathcal{E}_N( C_{q\rho^{\prime}(\beta)}^i)
\end{equation*}

\end{example}

\section{Evaluation of the Map $\mathcal{E}_N$}
\label{E section}
A computation bottleneck in the main algorithm of the CJP lies in evaluating $\mathcal{E}_N$ for elements $C \in \mathcal{A}_{\epsilon}.$
From the Theorem~\ref{main theorem} point of view it seems that this computation comes down to computing $\mathcal{E}_N(C^i_{q\rho^{\prime}(\beta)})$ where each $C^i_{q\rho^{\prime}(\beta)}$ is obtained from the quantum determinant.
However this naive approach is far from efficient. Our initial implementation using this interpretation could not practically perform calculations for some knots of index 10.
But, as mentioned in Remark~\ref{cody thm} the polynomial $C^1_{q\rho^{\prime}(\beta)}$ can be understood as a sum of the weights of walks on the braid. In this section we illustrate how this sum can be simplified to only require simple walks.
In other words, all non-simple walks can be ignored from the computation of $C^1_{q\rho^{\prime}(\beta)}$ entirely.

For simplicity we will be considering the evaluation of $\mathcal{E}_N(C)$ on an arbitrary element $C\in \mathcal{A}_{\epsilon}$.
We then focus on the special cases where $C$ is solely generated by walks. To this end, let 
\begin{equation*}
C=W_1+\cdots +W_t,
\end{equation*}
where each $W_i$ is an individual monomial in $\mathcal{A}_{\epsilon}$  for $1 \leq  i \leq t$.   
The polynomial evaluation of $\mathcal{E}_N(C)$ simplifies to:
\begin{equation}
\mathcal{E}_N(C)=\mathcal{E}_N(W_1)+\cdots+\mathcal{E}_N(W_t)
\end{equation}
following axiom $1$ of the evaluation map $\mathcal{E}_N$. Moreover, any monomial $W_i$ can be written as $W_i= \gamma W^{\prime}_i$ where $W^{\prime}_i$ is a word and $\gamma \in \mathcal{R}$. Hence by axiom 2 of the evaluation map $\mathcal{E}_N$ we have $$ \mathcal{E}_N(W_i)=\gamma \mathcal{E}_N(W^{\prime}_i). $$  Hence the problem of evaluating $\mathcal{E}_N(C)$ is reduced to the calculation of $\mathcal{E}_N$ on each word $W^{\prime}_i$ in $\mathcal{A}_{\epsilon}$. 
\subsection{Evaluation of Words in $\mathcal{A}_{\epsilon}$ }
We now describe the evaluation of an arbitrary word $W$ in $\mathcal{A}_{\epsilon}$.

\begin{lemma}
\label{explicit}
Let $W$ be an arbitrary word in $\mathcal{A}_{\epsilon}$ and let 
\begin{equation*}
\mathcal{V}(W)=( z, s_1, r_1, d_1, \cdots, s_u, r_u, d_u, s^{\prime}_1, r^{\prime}_1, d^{\prime}_1, \cdots, s^{\prime}_v, r^{\prime}_v, d^{\prime}_v)
\end{equation*}
be the word vector of $W$. Then
\begin{small}
\begin{equation*}
\mathcal{E}_N(W)=q^{z}  q^{\sum_{i=1}^u (N-1-d_i)r_i}  q^{ \sum_{j=1}^v (N-1)(-r_j)}  \prod_{i=1}^u \prod_{h=0}^{d_i-1}\left(1 - q^{N-1-r_i-h}\right) \prod_{j=1}^v  \prod_{l=0}^{d^\prime_j-1}  \left(1 - q^{r^{\prime}_j+l+1-N} \right).
\end{equation*}
\end{small}
\end{lemma}
\begin{proof}
Equation~\ref{quantum} implies that $W= q^{z} W^{\prime}$, where $z$ is an integer and $W^{\prime}$ is a right quantum word. Hence, by axiom $2$ of the evaluation map $\mathcal{E}_N$, \begin{equation*}
\mathcal{E}_N(W)=q^{z}\mathcal{E}_N(W^{\prime}).
\end{equation*}
Moreover, since $W^{\prime}$ is a right quantum word, $W^{\prime}=(W^{\prime})^{+} (W^{\prime}) ^{-}$. However, the words $(W^{\prime})^{+} $ and $(W^{\prime})^{-} $ are separable. Hence, axiom $3$ of $\mathcal{E}_N$ implies that: 
\begin{equation*}
\mathcal{E}_N\left(W^{\prime}\right)=\mathcal{E}_N\left((W^{\prime})^{+}\right)\mathcal{E}_N\left((W^{\prime})^{-}\right).
\end{equation*}
By identities~\ref{plus} and~\ref{min}, and axioms $4$ and $5$ of the function $\mathcal{E}_N$, we obtain the result.
\end{proof}

Lemma~\ref{explicit} nearly allows any word $W$ to be evaluated explicitly by simply counting the number of times each of the $a_i's$, $b_j's$ and $c_k's$ occurs in the word $W$ and applying the formula. The only problem is computing of the power $z,$ which  requires the algebra relations as we mentioned before. From relations~\ref{relation1},~\ref{relation2} and~\ref{relation3} we find that the $z$ can be computed given the following numbers :
\begin{enumerate}
\item The total of the number of times each $a^+_i$ appears to the left of each
$c^+_i$ in the word $W$.
\item The total of the number of times each $c^+_i$ appears to the left of each
$b^+_i$ in the word $W$.
\item The total of the number of times each $a^-_i$ appears to the left of each
$b^-_i$ in the word $W$.
\item The total of the number of times each $a^-_i$ appears to the left of each
$c^-_i$ in the word $W$.
\item The total of the number of times each $c^-_i$ appears to the left of each
$b^-_i$ in the word $W$.
\end{enumerate}
We will use the following convention to refer to the previous numbers. Let $W$ be a word and $x$ and $y$ be two letters in $W$. We denote by $B(x,y)(W)$
the total of the number of times each letter $x$ appears to the left of the letter $y$ in the word $W$.
When $W$ is clear from the context we will use the notation $B(x,y)$ instead. 
Using this convention and the relations~\ref{relation1},~\ref{relation2} and~\ref{relation3} we can write :
\begin{small}
\begin{eqnarray}
\label{z}
z=\sum_{i=1}^u B(a^+_i,c^+_i) -2  B(c^+_i,b^+_i) +  \sum_{j=1}^v 2B(a^-_j,b^-_j) -  B(a^-_j,c^-_j) +2 B(c^-_j,b^-_j)
\end{eqnarray}
\end{small}

Algorithm~\ref{BMEC} summarizes this process of calculating the word vector, $\mathcal{V}(W),$ of each word $W$ in $\mathcal{A}_{\epsilon}$.


\begin{algorithm}[H]
    \label{BMEC}
		\KwIn{ A word $W$ in $q's$, $a_i's$, $b_j's$ , $c_k's$ \;
        The number of crossings in a braid\;
        The sign of each crossing\;
        The $q$-value used in defining $\mathcal{A}_\epsilon$ .}
		\KwOut{ The $q$-coefficient of the word $W$ in the right quantum form\;
The list of exponents of $a_i's$\;
The list of exponents of $b_j's$\; 
The list of exponents of $c_l's$.} 
		\vspace{1mm}
Count the the number of $a_i's, b_j's,$ or $c_l's$ passed while traversing the word letter by letter\;
With each increment, keep track of how \cref{relation1,relation2} would change the power of $q$ for the word when putting the word into right quantum form as defined in section~\ref{sssec:RQWord}.

		\caption{Braid Monomial Exponent Counter (BMEC)}
\end{algorithm}

\begin{remark}
\label{rmk:qUnique}
Equation~\ref{z} and Lemma~\ref{explicit} can now be used to compute  $ \mathcal{E}_N(W)$ for any word $W$ in $\mathcal{A}_{\epsilon}$. In particular, let $W_1$ and $W_2$ be two words with the same letters in any order. Then the terms $s_1,r_1,d_1,\cdots,s_u,r_u,d_u,s^{\prime}_1,r^{\prime}_1,d^{\prime}_1,\cdots,s^{\prime}_v,r^{\prime}_v,d^{\prime}_v$ are identical in their word vectors, $\mathcal{V}(W_1)$ and $\mathcal{V}(W_2)$. In this case, we have that $\mathcal{E}_N(W_1)=\mathcal{E}_N(\gamma W_2)$ for some $\gamma\in\mathbb{Q}[q,q^{-1}]$.
\end{remark}

\subsection{Efficient Evaluation of $\mathcal{E}_N\left(\sum C^n\right)$ }
\label{Efficient Evaluation}
As before, let $C=W_1+W_2+\cdots +W_t$ in $\mathcal{A}_{\epsilon}$, where $W_i$ is a monomial for $1 \leq  i \leq t$. We denote by $\mathcal{W}_{\beta}^1=\{W_i\mid 1\leq i\leq t\}$ the set of monomials coming from walks in $C=C^1_{q\rho(\beta)}$. For words coming from stacks of walks, we denote the walks of stack height $s$ by $\mathcal{W}_\beta^s\coloneqq\{\prod_{j=1}^{s} W_{i_j}\mid 1\leq i_j\leq t\}$, so that each $W_{i_j}$ corresponds to the $i_j^{th}$ walk in $\mathcal{W}_{\beta}^1$. 
We are interested in the efficient computation of $\mathcal{E}_N\left(\sum C^n\right)$. As previously emphasized, this potentially infinite sum is a bottleneck for computing the algorithm. We employ two techniques for minimizing the number of the monomials in the set $\mathcal{W}_\beta\coloneqq\cup_s\mathcal{W}_\beta^s$ where we need to evaluate the function $\mathcal{E}_N$, which we describe next.

\subsubsection{Utilizing the Property of the CJP with Respect to the Mirror Image of Knots}
 Let $K$ and $\mr(K)$ be two knots that are mirror images of each other. Let $\beta$ and $\mr(\beta)$ be two braids with  $\hat{\beta}=K$ and $\mr(\hat{\beta}) =\mr( K)$. It is known that the colored Jones polynomial satisfies the property $J_{N,K}(q)=J_{N,\mr(K)}(1/q)$.  We can utilize this fact to reduce the number of computations using a simple strategy. By first computing the number of simple walks in $C_{q\rho^{\prime}(\beta)}$ and computing the number of simple walks in $C_{q\rho^{\prime}(\mr(\beta))}$ the program chooses to work with the one that has fewer simple walks to run the CJP computation.
 
Computing the simple walks of \emph{both} the braid and its mirror image, takes some additional time forcing the computation of the colored Jones polynomial to be slower than if only one set of simple walks were considered. However, this is only true for small choices of $N$. As mentioned earlier, the number of simple walks affects the performance of the algorithm more than any other parameter. By choosing the braid representation that minimizes this number, our algorithm significantly reduces the number of computations done as the number of colors increases. 


\subsubsection{The Duplicate Reduction Lemma}
Another large gain in computational efficiency is created by removing walks which have either zero contribution to the final sum or by deleting pairs of walks that have zero \emph{net} contribution to the final sum.
For computing $\mathcal{E}_N$ it is useful to divide the set of monomials from walk weights $\mathcal{W}_{\beta}^1=\{W_i \}_{i=1}^t$
 into the set of weights of simple and non-simple walks on the braid $\beta$. Due to~\cite[Lemma~4]{armond2014walks}, stated below in Lemma~\ref{lemma:DRL}, we only need to consider simple walks from the set $\mathcal{W}_\beta^1$. The non-simple walks occur in canceling pairs that have zero \emph{net} value when evaluated with the map $\mathcal{E}_N$.

\begin{lemma}[{Duplicate Reduction~\cite[Lemma~4]{armond2014walks}}]
\label{lemma:DRL}
~
\begin{enumerate}
\item \label{lemma:DRLp1}For any monomial, $W\in\mathcal{W}_\beta^1$, corresponding to the weight of a non-simple walk, 
 there is another monomial $W'\in\mathcal{W}_\beta^1$ whose weight evaluates to the negative of the original monomial.

\item \label{lemma:DRLp2} For any monomial, $W\in\mathcal{W}_\beta^s$, corresponding to the weight of a stack of simple walks, $W_{i_1},W_{i_2},\ldots,W_{i_s}$, with  $W=W_{i_1}W_{i_2}\cdots W_{i_s}$, where the stack of walks traverses the same point on $N$ different levels, the evaluation $\mathcal{E}_N(W)$ of that weight will be zero.
\end{enumerate}

\end{lemma}

Following Lemma~\ref{lemma:DRL}.\ref{lemma:DRLp1} we collect the monomials coming from the weights of non-simple walks :
\begin{definition}
\label{def:pairWs}
Two monomials $W_1$ and $W_2$, which have the same letters in any order are said to be \emph{paired} if $\mathcal{E}_N (W_1)=-\mathcal{E}_N(W_2)$. 
\end{definition}

Paired monomials help reduce the calculations needed for $\mathcal{E}_N(C^n)$ as follows. By  Lemma~\ref{lemma:DRL}.\ref{lemma:DRLp1} any non-simple walk, $W_i$ is \emph{paired} with a walk $W_j$ as in Definition~\ref{def:pairWs} so that $\mathcal{E}_N (W_i)+\mathcal{E}_N(W_j)=0$. Furthermore,  as we will discuss in Remark~\ref{rmk:ideal}, for paired walks $W_i, W_j$ and any walk $W_l$, $\mathcal{E}_N (W_lW_i)+\mathcal{E}_N(W_lW_j)=0=\mathcal{E}_N (W_iW_l)+\mathcal{E}_N(W_jW_l)$. So all non-simple walks can be ignored from the calculation of $\mathcal{E}_N(C)$ and furthermore in part~\ref{lemma:DRLp2} of the lemma, any stack of walks in $W_{i_1}W_{i_2}\cdots W_{i_s}\in\mathcal{W}_s$ can be ignored from the calculation as well, so long as one of the $W_{i_j}$ is non-simple. 

For monomials following the conditions of Lemma~\ref{lemma:DRL}.\ref{lemma:DRLp2} we define :
\begin{definition}
\label{def:ZeroE}
A monomial $W$ is said to be \emph{zero $N$-evaluated} if $\mathcal{E}_N(W)=0.$ 
\end{definition}

In example~\ref{ex:Fig8}, the non-simple walks $q^3c_2^+a_3^-a_4^+b_1^-c_3^-b_4^+$ and $q^2b_1^-c_3^-a_4^+c_2^+a_3^-b_4^+$ are paired in equation~\ref{eqn:exampleWalks}. Applying Definition~\ref{def:ZeroE} to these same walks for $N=2$ and $N=3$ we see that they are both zero 2-evaluated, but that neither is zero 3-evaluated. As with paired walks, given a zero $N$-evaluated walk $W_i$, where $\mathcal{E}_N(W_i)=0$, then $\mathcal{E}_N(W_lW_i)=0=\mathcal{E}_N(W_iW_l)$ for any walk $W_l\in\mathcal{W}_\beta$.

\begin{remark}
\label{rmk:ideal}
To be precise, given a braid $\beta$, we define $\mathcal{A}_\beta\coloneqq\Span_{\mathcal{R}}[\mathcal{W}_\beta]$, a subalgebra of $\mathcal{A}_{\epsilon}$. The algebra $\mathcal{A}_{\beta}$ represents exactly the minimal subalgebra of $\mathcal{A}_{\epsilon}$ containing the monomials needed by our algorithm to evaluate the colored Jones polynomial. In this subalgebra 
the $\mathcal{R}$-span of the set of zero $N$-evaluated words forms an ideal. 
Similarly, the span of the set of monomials $W$ corresponding to the weights of paired walks in $\mathcal{W}_\beta$ forms an ideal in $\mathcal{A}_\beta$. However, further information regarding these ideals does not further our computational aims so we omit the details.

\end{remark}

\begin{remark}
\label{rmk:nonsimple}
All non-simple walks are zero 2-evaluated. If $\mathcal{E}_2(W_i)=0$, then there exists a $W_j$ in $C$ so that both $\mathcal{E}_2(W_j)=0$ and $\mathcal{E}_N (W_i)+\mathcal{E}_N(W_j)=0$ and $\mathcal{E}_N (W_kW_i)+\mathcal{E}_N(W_kW_j)=0$ for all $N\geq2$, and any $W_k$ in $C$ as above. This allows the duplicate reduction lemma to be implemented with a single algorithm.
\end{remark}

Since non-simple walks will not be included in our computation we can ignore those walks completely from the determinant calculation. This is handled by Algorithms~\ref{The Simple Walk Calculator} and~\ref{Duplicate Reduction Lemma} simultaneously. 
Algorithm~\ref{The Simple Walk Calculator} focuses on removing nonsimple, paired walks.

\begin{algorithm}[h]
    \label{The Simple Walk Calculator}
		\KwIn{A braid sequence $\beta=((i_1,\epsilon_1),(i_2,\epsilon_2),\ldots,(i_m,\epsilon_m))$ such that $\hat{\beta}$ is a knot\;
        The $q$-value used in defining $\mathcal{A}_\epsilon$}
		\KwOut{ The sum of all \textit{simple walks} on strands $\{2, 3, \ldots, m\}$.} 
		\vspace{1mm}
Create an array of noncommutative variables $\mathcal{L}_{\epsilon}$\;
Compute the deformed Burau representation $\rho^{\prime}(\beta)$ defined in section~\ref{ssec:deformedB}\;
Create an array of the square $J-$submatrices of $\rho^{\prime}(\beta)$ for every $J\subseteq\{2,\ldots,m\}$\;
\tcp{$J$ is the index set of the original matrix $\rho(\beta)$.} 
Sum over all the $J-$submatrix determinants using equations~\ref{eqn:altSum} and~\ref{eqn:qdet}\;
\tcp{After each multiplication in equation~\ref{eqn:qdet} call Algorithm~\ref{Duplicate Reduction Lemma} with $N=2$ to remove all nonsimple walks}

		\caption{The Simple Walk Calculator (SWC) }
\end{algorithm}

Algorithm~\ref{Duplicate Reduction Lemma} assists in the implementation of algorithm~\ref{The Simple Walk Calculator}, with orders of magnitude speed increases for determining the reduced Burau representation.
It takes as input a list of walks $\mathcal{W}$ and returns the sublist containing the walks that do not $N$-evaluate to zero.

~

\begin{algorithm}[H]
    \label{Duplicate Reduction Lemma}
		\KwIn{A list of walks $\mathcal{W}$ \;
        The number of crossings in the braid $NCrs$ \;
        The number of colors $N\geq 1$}
		\KwOut{ A  sublist of $\mathcal{W}$ which contains all non paired walks and all non-zero $N$-evaluated walks.\; \tcp{For $N=2$ this returns all simple walks, by Remark~\ref{rmk:nonsimple}.}} 
		\vspace{1mm}

\While{$i\leq NCrs$ and $\mathcal{W}$ is nonempty}{
\For{$W\in \mathcal{W}$}{
Check the number of $a's$, $b's$ and $c's$ at crossing $i$\;
\If{$($ the number of \emph{a's} $+\max\{$the number of \emph{b's}, the number of \emph{c's}$\}) \geq N$}{Delete $W$ from the list $\mathcal{W}$}
}
}
\Return $\mathcal{W}$


\caption{The Duplicate Reduction Lemma Algorithm (DRL) }
\end{algorithm}

\section{Description of the Main Algorithm
}
The main algorithm takes as input a knot $K$, given as the closure of a braid $\beta$, a positive integer $N$ and $q$, which is either a commutative variable or a complex number. The braid $\beta$ can be chosen to be any braid with $\hat{\beta}=K$, though we recommend using the minimal forms described above for efficiency. 
The output of the main algorithm is the colored Jones polynomial $J_{N,K}(q)$. The algorithm starts by initiating the CJP to $1$. It then checks if the input braid is the empty braid, returning $1$ if this is the case. Otherwise the algorithm checks the number of  simple walks of the braid $\beta$ and its mirror image and retains whichever form had the minimal number of simple walks for the rest of the computation. After that the algorithm creates a while loop. This while loop calculates $\mathcal{E}_N( C_{q\rho^{\prime}(\beta)}^{i})$ for $i\geq 1 $ and adds those evaluations. The while loop terminates when $\mathcal{E}_N (C_{q\rho^{\prime}(\beta)}^{l})=0$ for some $l$. When this happens the algorithm exits the while loop and returns  $q^{(N-1)(\omega(\beta)-m+1)/2} \times \sum_{i=0}^{l-1}\mathcal{E}_N (C_{q\rho^{\prime}(\beta)}^{l}) $ after converting it into a Laurent polynomial.   

~

\begin{algorithm}[h]
    \label{The Colored Jones Polynomial}
		\KwIn{A braid sequence $\beta=((i_1,\epsilon_1),(i_2,\epsilon_2),\ldots,(i_k,\epsilon_k))$,
where $\epsilon_j = \pm 1$ and $\hat{\beta}$ is a knot\;
the commutative variable or complex number $q$\;
a positive integer $N\geq 1$.}
		\KwOut{ The colored Jones polynomial $J_{N,\hat{\beta}}(q)$. } 
		\vspace{1mm}
CJP$\coloneqq1$\tcp*[l]{ The evaluation of the empty braid, $\mathcal{E}_N\left(C_{q\rho^{\prime}(\beta)}^{0}\right)$, is $1$}

StackHeight$\coloneqq1$\tcp*[l]{ This is the exponent $i$ in $C_{q\rho^{\prime}(\beta)}^{i}$ }

Calculate SWC($\beta$) and SWC$\left(\mr(\beta)\right)$. If SWC$\left(\mr(\beta)\right)$ has fewer simple walks than SWC($\beta$), then reassign $q\coloneqq\frac{1}{q}, \beta\coloneqq\mr(\beta)$, and let  SW=SWC$\left(\mr(\beta)\right)$, otherwise let SW=SWC$(\beta)$\;
\tcp{SW is $C_{q\rho^{\prime}(\beta)}^1$  }

LoopDone $\coloneqq$ False\tcp*[l]{ This controls when the next while loop terminates}
 \While{ LoopDone $\neq$ True}{
{ \If {StackHeight = 1}
{WalkStack $\coloneqq$ SW\;}
\Else{
    WalkStack $\coloneqq$ $C_{q\rho^{\prime}(\beta)}^1\times$ MB($MEL$)\;
    \tcp{This state can be reached only after list $MEL$ is defined.}
    \tcp{Here we call MonomialBuilder on each monomial in WalkStack in order to efficiently compute the next right quantum form $C_{q\rho^{\prime}(\beta)}^{i}$ - here $i=StackHeight$.}
    
    WalkStack $\coloneqq$ DRL(WalkStack)\tcp*[l]{Minimize the number of words in  $C_{q\rho^{\prime}(\beta)}^{i}$}}}
\If{WalkStack $\neq  0$ }{
     Run the function BMEC on WalkStack and store that in a list  $MEL$\;
     \tcp{ $MEL= \{ \tilde{q}, \{s\},\{r\},\{d\} \}$, where each list in $MEL$ contains $\tilde{q}$, the $q$-coefficient of a walk and lists $ \{s\},\{r\},\{d\}$ encoding the number of $a_i$'s, $b_j$'s, and $c_k$'s in that same word}
 WalkWeights $\coloneqq \mathcal{E}_N\left(MEL\right)$\tcp*[l]{ Applying the formula in Lemma~\ref{explicit}}
 \tcp{Here $MEL$ is the minimal data from  $C_{q\rho^{\prime}(\beta)}^i$ needed by $\mathcal{E}_N$.}
 \If{WalkWeights = 0 }{LoopDone $\coloneqq$ True\; }
 \Else{
 CJP $\coloneqq$ CJP + WalkWeights\;
 StackHeight++\;}
  }
  \Else{
    LoopDone $\coloneqq$ True\;
  }
}
\Return $q^{(N-1)(\omega(\beta)-m+1)/2} \times $ CJP\tcp*[l]{Simplifying this promotes utility}

		\caption{The Colored Jones Polynomial}
\end{algorithm}

\section{Implementation and Performance}
In this section we give a brief description of our Mathematica implementation of the algorithms before we discuss the performance of algorithm~\ref{The Colored Jones Polynomial} on the knot table.
\subsection{Implementation}
For our implementation we used Mathematica~\cite{Mathematica} to handle the symbolic computation of the algorithm. For noncommutative multiplication in the algebra $\mathcal{A}_{\epsilon}$, we used NCAlgebra package for Mathematica~\cite{NCalgebra}.
In particular, we relied on the NCPoly data structure, which allowed the majority of computations to be done using integer arithmetic rather than symbolic calculations, providing significant speed gains.
Additionally we take advantage of Mathematica's memoization 
techniques and functional programming structure for additional efficiency gains in algorithms~\ref{MB} and ~\ref{The Simple Walk Calculator}, and with helper functions to algorithm~\ref{BMEC}.
In algorithm~\ref{Duplicate Reduction Lemma} the inner \textbf{for} loop is replaced with code that utilizes Mathematica's pattern recognition features, rather than using a constantly reindexing list suffering frequent object removal.
Intensive pieces of code are reduced to C-compiled parallelized functions where possible. Thus the code contains two implementations of Algorithms~\ref{MB},~\ref{BMEC},~and~\ref{Duplicate Reduction Lemma}, one using C-code and a commented out version using pure Mathematica code, for use if your computer does not include a C-compiler.

In implementing Algorithm~\ref{The Colored Jones Polynomial} the return step includes two features.
The community is often interested in evaluating the color Jones polynomial at roots of unity, so before returning the polynomial it evaluates at $q$ if a complex variable was supplied rather than an indeterminant.
Additionally, the initial output is often a dense collection of sums and products. We add two steps into the code to render the data into Laurent polynomials before output. All the data collected in the charts below includes this extra step of simplifying the information into this form. The additional requirement to return Laurent polynomials increases run times by up to an order of magnitude. We have chosen to include this in our performance evaluations as we believe most users will wish to render the data into a simpler form before using it.


The software for the CJP polynomial is available for public use  and currently can be downloaded from the GitHub repository: \href{http://github.com/jsflevitt/color-Jones-from-walks}{github.com/jsflevitt/color-Jones-from-walks}. 

\subsection{Performance}
In order to test the efficiency of our method we performed several tests on the knot table with knots whose crossing numbers are less than or equal to 12. Our tests were performed on a 3.20 GHz Intel Core i5 with 8.0
Gb of memory on the macOS platform. As we mentioned earlier the computational time of algorithm~\ref{The Colored Jones Polynomial} relies mainly on the number of simple walks of the input knot. For this reason we test this algorithm with respect to the number of simple walks. Specifically, recall that algorithm~\ref{The Colored Jones Polynomial} computes the minimal number of the simple walks between the knot and its mirror and then it performs the computations with the knot which has fewer number of simple walks. For this reason, the running time of~\ref{The Colored Jones Polynomial} is compared against the number of simple walks between the knot and its mirror image. This performance is shown in Figure~\ref{p1} for the all knots in the knots table with number of crossings less than or equal 12. From Figure~\ref{p1} one can observe that there are  several knots that share the same number of simple walks.



\begin{figure}[h]
  \centering
   {\includegraphics[scale=0.73]{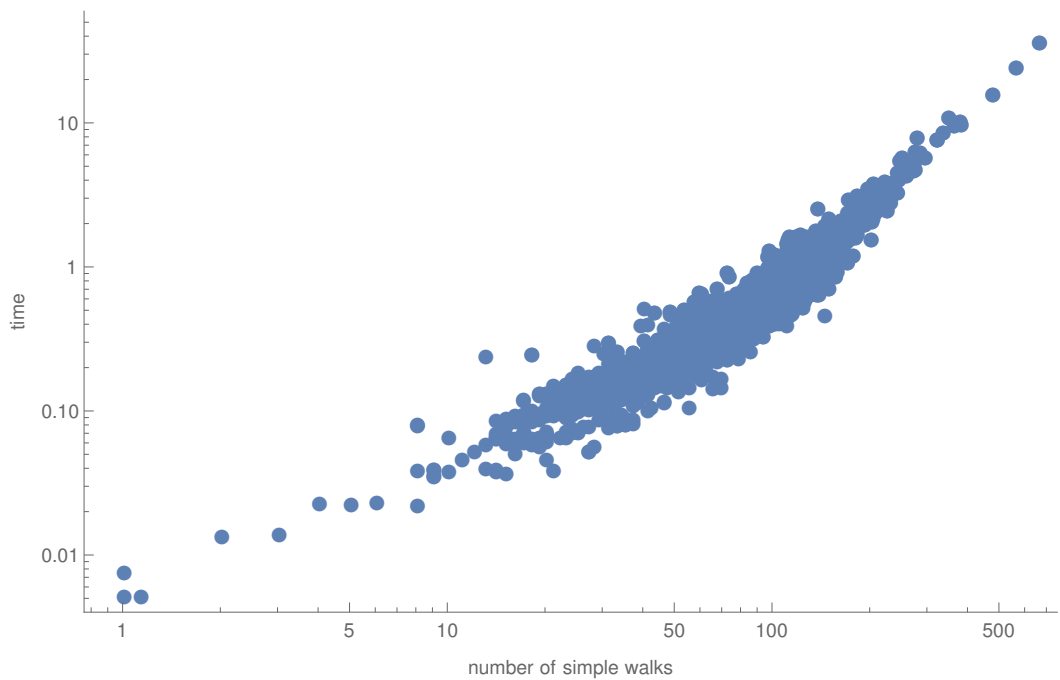}
   \caption{Algorithm \ref{The Colored Jones Polynomial} performance with respect to the number of simple walks in the braid. Here the computations are done with color N=2 and displayed with a Log-Log scale.}
  \label{p1}}
\end{figure}

The correlation between the number of crossings of a knot and the number of simple walks is not immediate from the definition. 
Moreover, Algorithm~\ref{Duplicate Reduction Lemma} makes non-linear reductions to the number of final walks needed to compute the colored Jones polynomial in Algorithm \ref{The Colored Jones Polynomial}. To obtain a better understanding of the growth of the number of simple walks as the number of crossing increases we give Figure~\ref{p2} where the number of simple walks needed to compute the colored Jones polynomial is plotted with respect to the number of crossings. Figure~\ref{p2} suggests that average of number of simple walks is bounded by $\mathcal{O}(k^2)$, where $k$ is the number of crossings. Although proving this claim requires a more thorough analysis, starting with understanding how to find the minimal braid word for any knot $K$.
\begin{figure}[h]
  \centering
   {\includegraphics[scale=0.7]{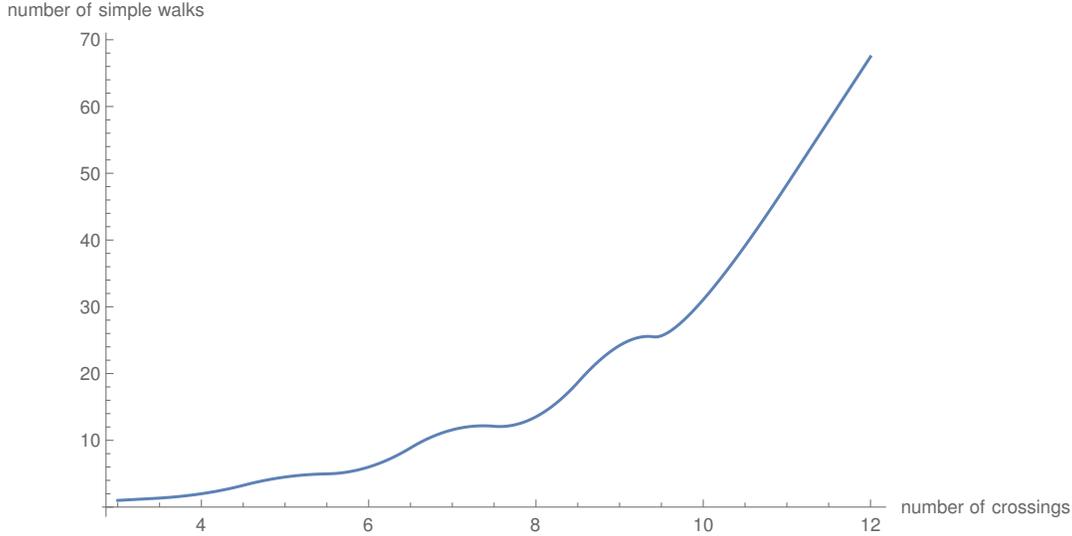}
    \caption{The growth of the average of the number of simple walks with respect to the number of crossings.}
  \label{p2}}
\end{figure}

The performance of our algorithm as the number of color increases is shown in Figure~\ref{p3}. Here we compute the colored Jones polynomial for each color $N$ between $2$ and $7$ three times for each knot and compare the average running time with respect to the number of simple walks. We show the performance of the algorithm on individual knots as we increase the number of colors in Figure \ref{p4} and elaborate on this when we compare our algorithm with the algorithm provided in KnotTheory package~\cite{KA}. 

\begin{figure}[h]
  \centering
   {\includegraphics[scale=0.73]{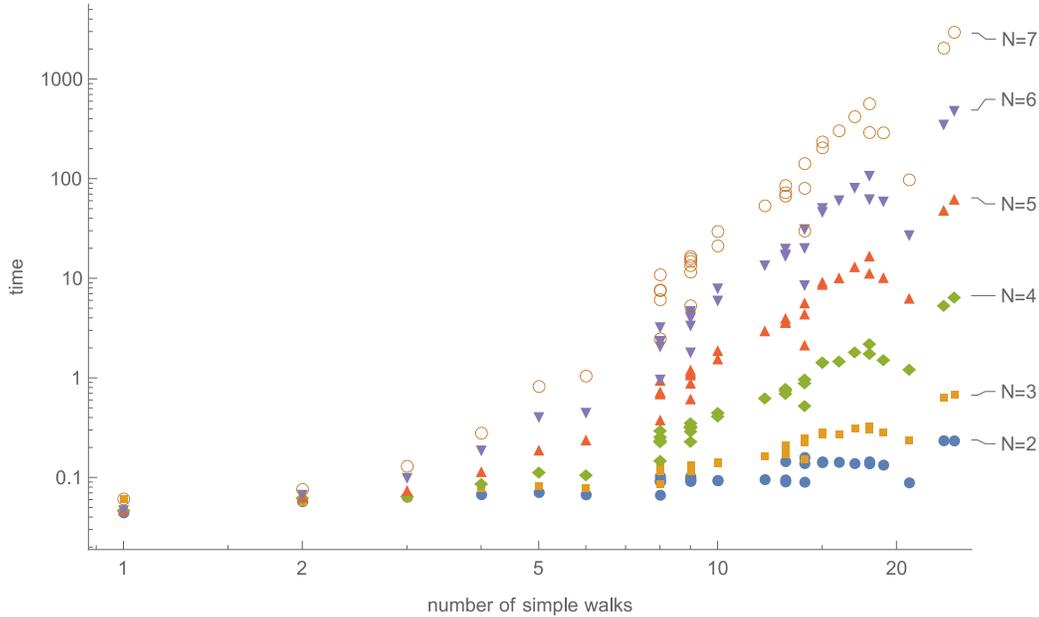}
    \caption{The growth of the running time, measured in seconds, with respect to the number of simple walks. Here the computation are shown for colors 2,3,4,5,6 and 7.}
  \label{p3}}
\end{figure}

\subsection{Number of Walks and The Duplicate Reduction Lemma}
To show the impact the Algorithm~\ref{Duplicate Reduction Lemma} on the running time we conducted several tests. We run our experiments on the first on all knots in the knot table with crossings less than or equal to 9. For all these knots we computed the number of  walks in in $C^{1}_{q \rho^{\prime}(\beta)}$. Having the number of simple walks in $C^{1}_{q \rho^{\prime}(\beta)}$ minimal is critical because our method computes $C^{n}_{q \rho^{\prime}(\beta)}$ which for all $n \geq 1 $ such that $C^{n}_{q \rho^{\prime}(\beta)} \neq 0$. Our comparison goes as follows. For each knot with number of crossing between 3 and 9 we compute the number of walks in $C^{1}_{q \rho^{\prime}(\beta)}$ in two different ways : (1) with the utilization of Algorithm~\ref{Duplicate Reduction Lemma} and (2) without using that Algorithm. Figure~\ref{impact_using_lemma} shows the impact of using Algorithm~\ref{impact_using_lemma} on the total number of walks $C^{1}_{q \rho^{\prime}(\beta)}$ that is needed to compute the colored Jones polynomial.

\begin{figure}[h]
  \centering
   {\includegraphics[scale=0.73]{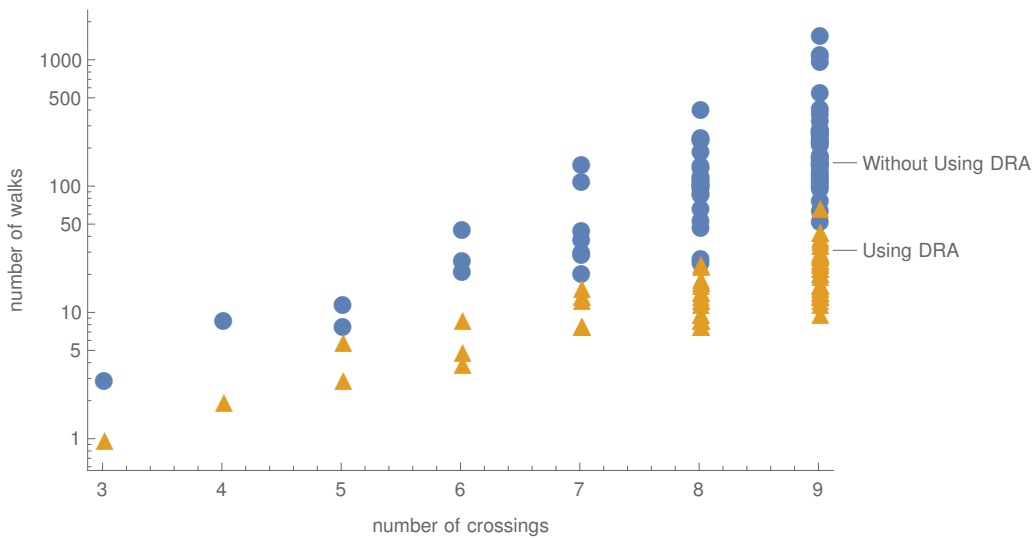}
    \caption{The effect of using Algorithm~\ref{Duplicate Reduction Lemma} on the total the number of walks in $C^{1}_{q \rho^{\prime}(\beta)}$. Circular dots represent the number of walks in $C^{1}_{q \rho^{\prime}(\beta)}$ without using the DRA algorithm while the triangular dots represent the number of number of walks of $C^{1}_{q \rho^{\prime}(\beta)}$ using the DRA algorithm.  }
  \label{impact_using_lemma}}
\end{figure}
In Figure~\ref{impact_using_lemma_time} we did we show impact of utilization Algorithm ~\ref{Duplicate Reduction Lemma} on the running time. Note that using Algorithm~\ref{Duplicate Reduction Lemma} impacts the running time of the Algorithm~\ref{The Colored Jones Polynomial} by an order of magnitude. In our experimentation some knots with 9 crossings number did not even run on our machine without the utilization of Algorithm~\ref{Duplicate Reduction Lemma} in the Algorithm~\ref{The Colored Jones Polynomial}.  

\begin{figure}[h]
  \centering
   {\includegraphics[scale=0.73]{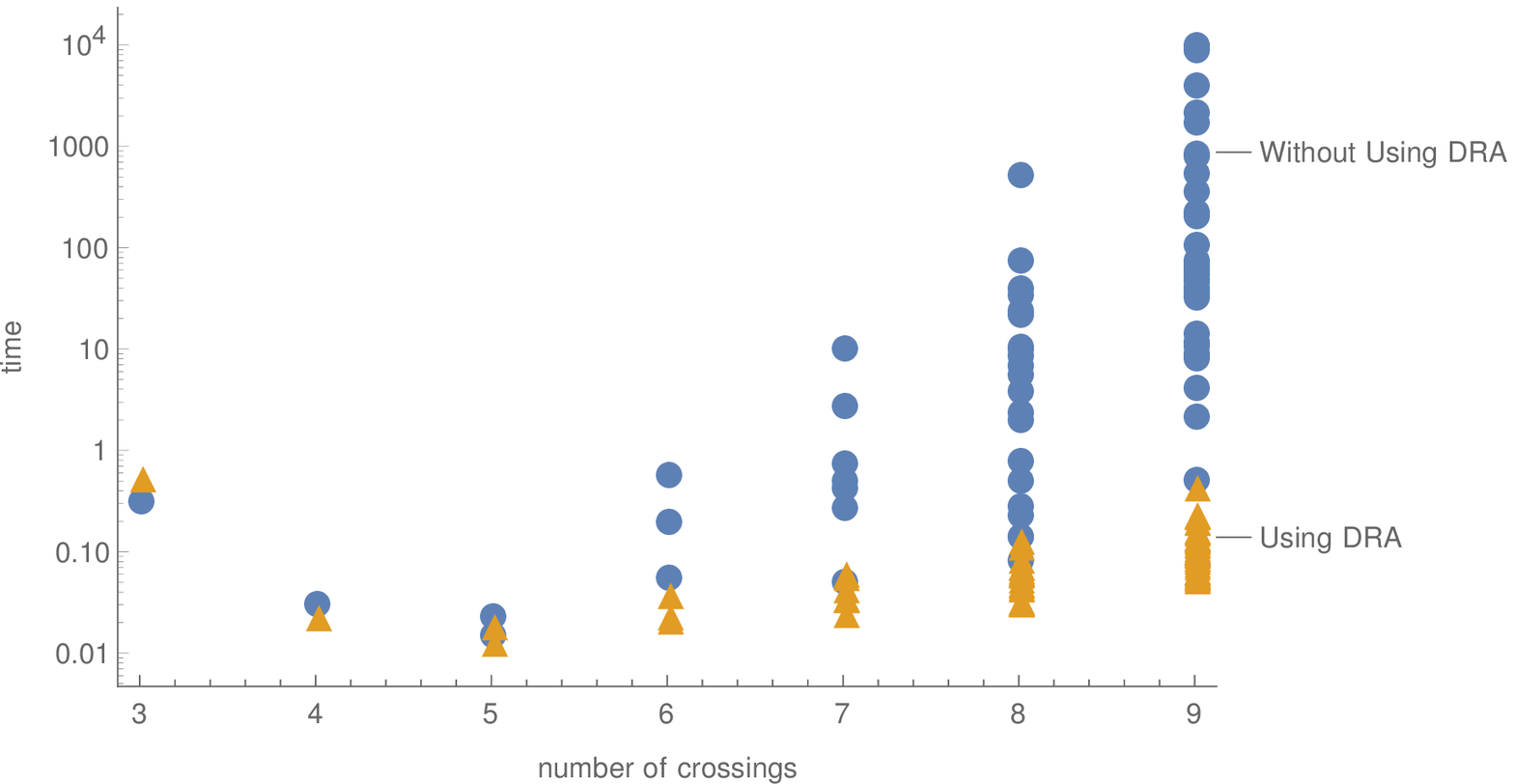}
    \caption{The effect of using Algorithm~\ref{Duplicate Reduction Lemma} on the total running time of Algorithm~\ref{The Colored Jones Polynomial}. Circular dots represent the running time of the algorithm ~\ref{The Colored Jones Polynomial} without using Algorithm ~\ref{Duplicate Reduction Lemma} while the triangular dots represent the number running time of Algorithm ~\ref{The Colored Jones Polynomial} using the DRA algorithm. }
  \label{impact_using_lemma_time}}
\end{figure}


\subsection{Comparing the running time with the KnotTheory Package}
We also run our algorithm against the implementation of the colored Jones polynomial provided in KnotTheory package \cite{KA}. Figure~\ref{p4} shows the running time comparison between our method and the KnotTheory package REngine method for the first 28 knots in the knot table. 

The circular dots represent the  KnotTheory Package running time while the triangular dots represent our algorithm's running time. The figure shows that our method is faster by an order of magnitude for most knots shown in the figure. As a final note, we have yet to overflow memory while running our algorithm, while the REngine calculation occasionally crashed for this reason preventing a fuller comparison.

\begin{figure}[h]
  \centering
   {\includegraphics[scale=0.415]{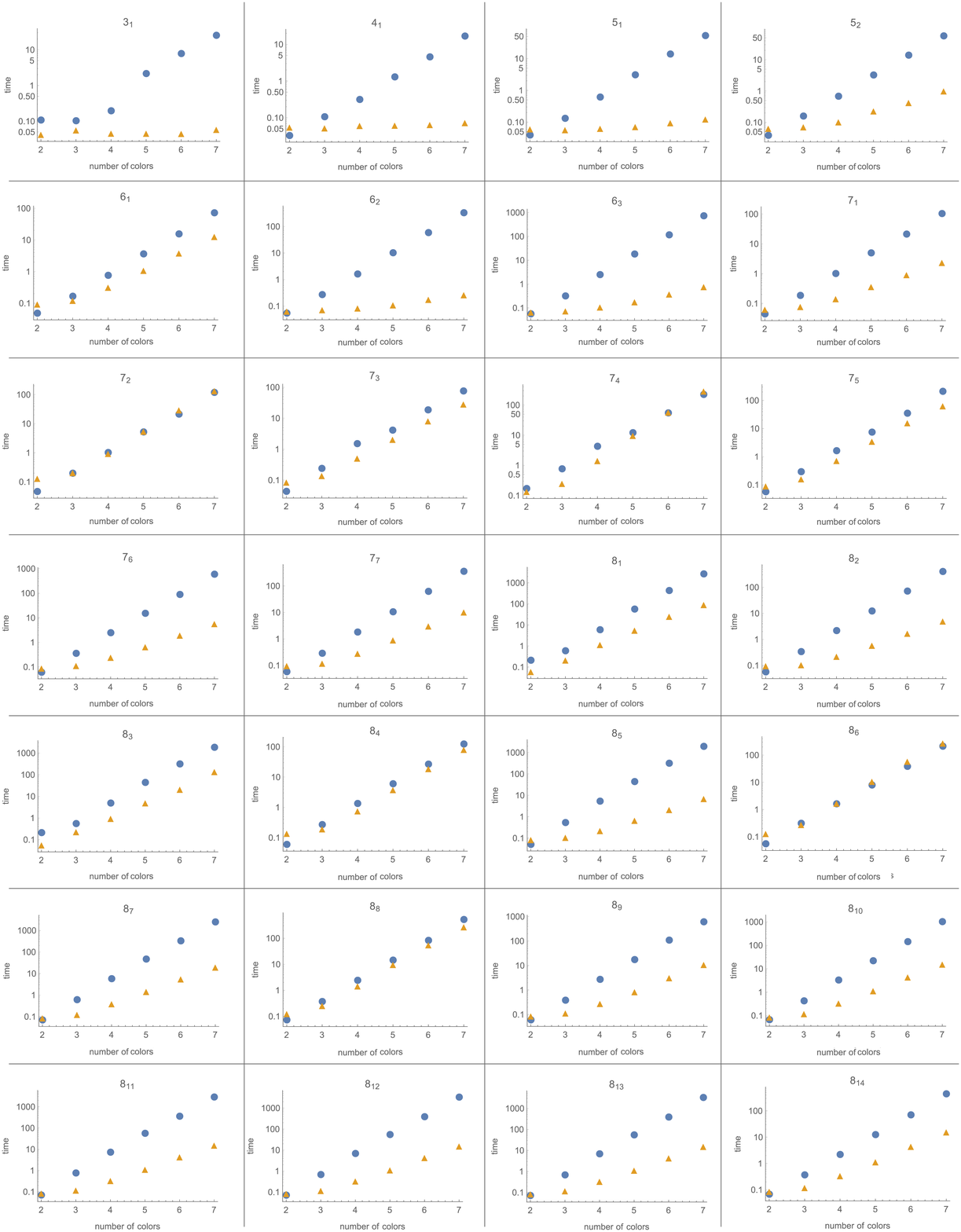}
    \caption{A comparison between the running time of the CJP algorithm implemented in the KnotTheory Package~\cite{KA}, represented by circles, and our algorithm, represented by triangles, on the first 28 knots in the knot table.
The running time is in seconds and the $x-$axis represents the number of colors. }
  \label{p4}}
\end{figure}

\section{Acknowledgment}
The authors would like to thank Cody Armond and Kyle Istvan for many useful conversations while writing this paper. We also would like to thank Matt Hogancamp for several suggestions that improved the paper's clarity.  The second author was partially supported by the National 
Science Foundation (grant no. DMS-1764210). Finally we thank Iris Buschelman, Ina Petkova and Katherine Lee Pierret Perkins for their assistance.

\bibliographystyle{plain}

\bibliography{cjp.bib}

\end{document}